\documentclass[reqno,12pt]{amsart}
\usepackage[T1]{fontenc}
\usepackage{enumerate}
\theoremstyle{plain}
\newtheorem{thm}{Theorem}

\newtheorem{corollary}{Corollary}

\begin{document}
\title{On a second-order rational difference equation and a rational system}

\author[Gabriel Lugo]{Gabriel Lugo}
\address{Department of Mathematics, University of Rhode Island,Kingston, RI 02881-0816, USA;}
\email{glugo@math.uri.edu}
\author[Frank J. Palladino]{Frank J. Palladino}
\address{Department of Mathematics, University of Rhode Island,Kingston, RI 02881-0816, USA;}
\email{frank@math.uri.edu}
\date{January 25, 2012}

\begin{abstract}
\noindent We give a complete description of the qualitative behavior of the second-order rational difference equation \#166. 
We also establish the boundedness character for the rational system in the plane \#(8,30).
\end{abstract}
\maketitle

\section{Introduction}
In their book \cite{klbook}, Kulenovi\'c and Ladas initiated a systematic study of the general second-order rational difference equation,
\begin{equation}\label{b}
x_{n+1}=\frac{\alpha +\beta x_{n} + \gamma x_{n-1}}{A+ B x_{n} + C x_{n-1}},\quad n=0,1,\dots ,
\end{equation}
with nonnegative parameters so that $A+B+C >0$ and nonnegative inital conditions chosen to avoid division by zero. A main feature of this study was the subdivision of the problem into a large number of special cases.
The study of these special cases has attracted a great deal of attention in the literature. A large amount of work has been directed toward developing a complete picture of the qualitative behavior of the difference equation \eqref{b}.
A detailed account of the progress up to 2007 can be found in \cite{acl2p1} and \cite{acl2p2}. A more recent account of the subsequent progress up to 2009 can be found in \cite{basumerino}.
According to Ref. \cite{basumerino}, there remain only two special cases of \eqref{b} for which the qualitative behavior has not been established yet. 
However, the authors of Ref. \cite{basumerino} claim that a change of variables found in Ref. \cite{klr} reduces the special case \#141 to the special case \#66 which was resolved in \cite{merino}.
After carefully reading \cite{klr} we could not find the purported change of variables. So, as far as we know, there remain three special cases of \eqref{b} for which the qualitative behavior has not been established yet.
These three remaining special cases are the cases numbered 68, 141 and 166 in the numbering system given in \cite{clbook}. The special cases \#68 and \#141 are the two subcases of the following second-order rational difference equation 
$$x_{n+1}=\frac{\alpha + x_{n}}{A+ x_{n}+ Cx_{n-1}},\quad n=0,1,\dots ,$$
with $A\geq 0$, all other parameters positive, and nonnegative initial conditions.
The special case \#166 is the second-order rational difference equation 
\begin{equation}\label{a}
x_{n+1}=\frac{\alpha+\beta x_{n}+x_{n-1}}{A+x_{n-1}},\quad n=0,1,\dots ,
\end{equation}
with positive parameters and nonnegative initial conditions.
In this article, we prove that in all ranges of positive parameters the unique equilibrium of the difference equation \eqref{a} is globally asymptotically stable. Thus, the special cases \#68 and \#141 are the only remaining cases of \eqref{b} for which the qualitative behavior has not been established yet.\par
More recently, in \cite{cklm}, Camouzis, Kulenovi\'c, Ladas, and Merino have initiated a systematic study of the general rational system of difference equations in the plane,
\begin{equation}\label{system}
x_{n+1}=\frac{\alpha+\beta x_{n} + \gamma y_{n}}{A+B x_{n} + C y_{n}},\quad y_{n+1}=\frac{p+\delta x_{n} + \epsilon y_{n}}{q+D x_{n} + E y_{n}},\quad n=0,1,2,\dots ,
\end{equation}
with nonnegative parameters and nonnegative inital conditions chosen to avoid division by zero. According to Ladas, there remain only two special cases of \eqref{system} for which the boundedness character has not been established yet. 
These two remaining special cases are the cases numbered $(6,25)$ and $(8,30)$ in the numbering system given in \cite{cklm}. The special case \#$(6,25)$ is the system 
$$x_{n+1}=\frac{x_{n}}{y_{n}},\quad y_{n+1}= x_{n} + \epsilon y_{n},\quad n=0,1,2,\dots ,$$
with $\epsilon > 0$ and nonnegative initial conditions.
The special case \#$(8,30)$ is the system 
\begin{equation}\label{8_30}
x_{n+1}=\frac{y_{n}}{x_{n}}, \quad y_{n+1}=\frac{\alpha + \gamma y_{n}}{x_{n}+y_{n}}\quad n=0,1,\dots ,
\end{equation}
with positive parameters and nonnegative initial conditions.
In this article, we prove that in all ranges of positive parameters every solution of the difference equation \eqref{8_30} is bounded. Thus, the special case \#$(6,25)$ is the only remaining case of \eqref{system} for which the boundedness character has not been established yet.\par

\section{Equation \#166}

\begin{thm}
The unique equilibrium of the difference equation \eqref{a} is globally asymptotically stable. 
\end{thm}
\begin{proof}
The proof will proceed in three cases. The first case will be the case where $A\leq \alpha$, which was proved in \cite{klbook}. For the reader's convenience, we will restate the proof here. 
Notice that for $A\leq \alpha$,
$$x_{n+1}=\frac{\alpha+\beta x_{n}+x_{n-1}}{A+x_{n-1}}=\frac{A-A+\alpha+\beta x_{n}+x_{n-1}}{A+x_{n-1}}=1+\frac{\alpha-A+\beta x_{n}}{A+x_{n-1}}.$$
Via the change of variables $x_n=1+z_n$, this difference equation reduces to the following equation.
$$z_{n+1}=\frac{\alpha-A+\beta+\beta z_{n}}{A+1+z_{n-1}}.$$
It was shown in \cite{merino} that the unique equilibrium is globally asymptotically stable for the above equation, thus the unique equilibrium is globally asymptotically stable for the difference equation \eqref{a} in this case.
Now we will address the case where $A>\alpha$ and $A\leq \beta + \alpha$. We will begin by showing that the interval $[1,\infty )$ is an invariant attracting interval.
First notice that 
\begin{equation}\label{partial}
\frac{\partial }{\partial x}\left(\frac{\alpha+\beta y+x}{A+x}\right)=\frac{A-\alpha-\beta y}{{\left(A+x\right)}^2}.
\end{equation}

Suppose $x_n\geq 1$, then 
$$x_{n+1}=\frac{\alpha +\beta x_n +x_{n-1}}{A+x_{n-1}}\geq \frac{\alpha +\beta +x_{n-1}}{A+x_{n-1}},$$
and, due to the fact that $A\leq \beta + \alpha$, 
$$x_{n+1}\geq \frac{\alpha +\beta +x_{n-1}}{A+x_{n-1}}\geq 1.$$
It follows by induction that whenever $x_N\geq 1$, then $x_n\geq 1$ for all $n\geq N$. Thus, $[1,\infty )$ is an invariant interval.
Let $\{x_{n}\}^{\infty}_{n=1}$ be a nonnegative solution to the difference equation \eqref{a}. Assume that $x_{n}\not\in [1,\infty )$ for all $n\in\mathbb{N}$.
Then $x_{n}< \frac{A-\alpha}{\beta}$ for all $n\in\mathbb{N}$, since if $x_{N}\geq \frac{A-\alpha}{\beta}$ for some $N\in\mathbb{N}$, then
$$x_{N+1}=\frac{\alpha +\beta x_N +x_{N-1}}{A+x_{N-1}}\geq \frac{A +x_{N-1}}{A+x_{N-1}}=1,$$
yielding a contradiction to our prior assumption.
Since $x_{n}< \frac{A-\alpha}{\beta}$ for all $n\in\mathbb{N}$ we get via \eqref{partial},
\begin{equation}\label{inequality_1}
x_{n+1}=\frac{\alpha +\beta x_n +x_{n-1}}{A+x_{n-1}}\geq \frac{\alpha +\beta x_n }{A},\quad n\in\mathbb{N}.
\end{equation}
So, under these assumptions, our solution satisfies the difference inequality \eqref{inequality_1}. Now, in the case where $A\leq \beta$ any solution which satisfies the difference inequality \eqref{inequality_1} must be an unbounded solution, contradicting the assumption $x_{n}\not\in [1,\infty )$ for all $n\in\mathbb{N}$.
On the other hand, if $A>\beta$ then, applying Theorem 3 from \cite{fpiteration} or similar results, for each $\epsilon>0$ there exists an $N_{\epsilon}$ so that $x_{n}\geq \frac{\alpha}{A-\beta}-\epsilon$ for all $n\geq N_{\epsilon}$. Since $A\leq \beta + \alpha$ in this case, $\frac{\alpha}{A-\beta}\geq 1$. So, in this case, for each $\epsilon>0$ there exists an $N_{\epsilon}$ so that $x_{n}\in [1-\epsilon, 1)$ for $n\geq N_{\epsilon}$. In other words, $x_{n}\rightarrow 1$.
Thus, in this case, an arbitrary solution either converges to $1$, or enters the invariant interval $[1,\infty )$.
Now, for a solution in $[1,\infty )$, we may make the change of variables $x_n=1+z_n$ reducing the equation as follows.
$$z_{n+1}=\frac{\alpha+1+z_{n-1}+\beta\left(1+z_{n}\right)}{A+1+z_{n-1}}-1=\frac{\alpha-A+\beta+\beta z_{n}}{A+1+z_{n-1}}.$$
Now, if $\alpha +\beta > A$, then this reduced equation was resolved in \cite{merino}, where it was shown that the unique equilibrium is globally asymptotically stable for the above equation. In the very special subcase where $\alpha +\beta = A$, the above equation may be rewritten as,
$$z_{n+1}=\frac{(A-\alpha ) z_{n}}{A+1+z_{n-1}}.$$
For the above equation, every solution converges to zero. Thus, $x_{n}\rightarrow 1$ in this very special case. 
So, we have shown that in the case where $A>\alpha$ and $A\leq \beta + \alpha$ that every nonnegative solution of the difference equation \eqref{a} converges to the unique positive equilibrium.
The final case we must consider is the case where $A> \beta +\alpha$. We will begin our consideration of this case, by proving that the interval $[0,\frac{A-\alpha}{\beta}]$ is invariant.
Suppose that $x_{n}\leq \frac{A-\alpha}{\beta}$, then 
$$x_{n+1}=\frac{\alpha +\beta x_n +x_{n-1}}{A+x_{n-1}}\leq 1 < \frac{A-\alpha}{\beta}.$$
Now we show that $[0,\frac{A-\alpha}{\beta}]$ is attracting. Assume, for the sake of contradiction, that $x_{n}\not\in [0,\frac{A-\alpha}{\beta}]$ for all $n\in\mathbb{N}$. 
Under this assumption, since $x_{n}> \frac{A-\alpha}{\beta}$, we may use \eqref{partial} to obtain the following difference inequality.
$$x_{n+1}=\frac{\alpha +\beta x_n +x_{n-1}}{A+x_{n-1}}\leq \frac{\alpha +\beta x_n}{A}\quad n\in\mathbb{N}.$$
Thus, applying Theorem 2 from \cite{fpiteration} or similar results, for each $\epsilon>0$ there exists a $N_{\epsilon}$ so that $x_{n}\leq \frac{\alpha}{A-\beta}+\epsilon$ for all $n\geq N_{\epsilon}$. Now, since $\frac{\alpha}{A-\beta}<1<\frac{A-\alpha}{\beta}$, every solution with these properties must eventually enter the interval $[0,\frac{A-\alpha}{\beta}]$, contradicting our assumption that this does not occur.
So, we have shown that in this case the interval $[0,\frac{A-\alpha}{\beta}]$ is an invariant interval which every solution must eventually enter. 
Since our difference equation is nondecreasing with respect to each argument and has a unique nonnegative equilibrium in this interval, the m-M theorem, see \cite{mm} and \cite{mmlp}, implies that the unique equilibrium $\bar{x}$ is globally asymptotically stable in this case. The unique equilibrium of equation \#166 is well known to be locally asymptotically stable in all cases, see \cite{klbook} for local stability of the cases we have not yet shown.
\end{proof}

\section{The boundedness character of the special case \#$(8,30)$}
Now we present the boundedness character of the following system numbered \#$(8,30)$ in the numbering system developed in \cite{cklm}.
$$x_{n+1}=\frac{y_{n}}{x_{n}}, \quad y_{n+1}=\frac{\alpha + \gamma y_{n}}{x_{n}+y_{n}}\quad n=0,1,\dots ,$$
with nonnegative parameters and nonnegative initial conditions.
It turns out that the $x_{n}$ component of the system \#$(8,30)$ can be reduced to the difference equation
\begin{equation}\label{8_30_1}
x_{n}=\left(\frac{1}{x_{n-1}x_{n-2}}\right)\frac{\alpha + \gamma x_{n-1}x_{n-2}}{1+x_{n-1}},\quad n\geq 2 ,
\end{equation}
through algebraic identities. This reduction proceeds as follows.
The first equation of the system \#$(8,30)$ gives us
$$y_{n}=x_{n+1}x_{n},\quad n\geq 0.$$
Substituting this in for $y_{n}$ in the second equation gives us
$$y_{n+1}=\frac{\alpha + \gamma x_{n+1}x_{n}}{x_{n}+x_{n+1}x_{n}},\quad n\geq 0.$$
Substituting this in for $y_{n}$ in the first equation gives us
$$x_{n+2}=\left(\frac{1}{x_{n+1}x_{n}}\right)\frac{\alpha + \gamma x_{n+1}x_{n}}{1+x_{n+1}},\quad n\geq 0.$$
This yields the difference equation \eqref{8_30_1}. So from now on it suffices to show that every solution of \eqref{8_30_1} is bounded.
\begin{thm}\label{thm3}
Every solution is bounded for the rational difference equation \eqref{8_30_1} with positive parameters and positive initial conditions.
\end{thm}
\begin{proof}
Recall our difference equation \eqref{8_30_1},
$$x_{n}=\left(\frac{1}{x_{n-1}x_{n-2}}\right)\frac{\alpha + \gamma x_{n-1}x_{n-2}}{1+x_{n-1}},\quad n\geq 2 .$$
Iterating with respect to the leftmost $x_{n-1}$ term in the denominator we get,
\begin{equation}\label{8_30_2}
 x_{n}=x_{n-3}\left(\frac{1+x_{n-2}}{\alpha + \gamma x_{n-2}x_{n-3}}\right)\frac{\alpha + \gamma x_{n-1}x_{n-2}}{1+x_{n-1}},\quad n\geq 3 .
\end{equation}
Equation \eqref{8_30_2} yields the following inequality,
\begin{equation}\label{8_30_3}
x_{n}< \left(\frac{x_{n-3}}{\alpha}+\frac{1}{\gamma}\right)\frac{\alpha + \gamma x_{n-1}x_{n-2}}{1+x_{n-1}},\quad n\geq 3.
\end{equation}
Iteration gives us the following inequality,
$$\frac{1}{\alpha+\gamma x_{n-2}x_{n-3}}=\frac{1}{\alpha + \gamma x_{n-3}\left(\frac{1}{x_{n-3}x_{n-4}}\right)\frac{\alpha + \gamma x_{n-3}x_{n-4}}{1+x_{n-3}}}$$ $$< \frac{1}{\alpha + \gamma x_{n-3}\left(\frac{1}{x_{n-3}x_{n-4}}\right)\frac{\gamma x_{n-3}x_{n-4}}{1+x_{n-3}}}
=\frac{1}{\alpha+\frac{\gamma^{2}x_{n-3}}{1+x_{n-3}}}.$$
Using the above Equation \eqref{8_30_2} yields
\begin{equation}\label{8_30_4}
 x_{n}< x_{n-3}\left(\frac{1}{\alpha+\frac{\gamma^{2}x_{n-3}}{1+x_{n-3}}}+\frac{1}{\gamma x_{n-3}}\right)\frac{\alpha + \gamma x_{n-1}x_{n-2}}{1+x_{n-1}},\quad n\geq 3 .
\end{equation}
The proof proceeds by contradiction. Suppose that there is an unbounded solution $x_{n}$. Take a subsequence $x_{n_{i}}\rightarrow\infty$ so that $x_{n_{i}}>x_{n_{i}-3}$ for all $i\in\mathbb{N}$.
Then equation \eqref{8_30_1} gives us that $$x_{n_{i}-1}x_{n_{i}-2}\rightarrow 0,$$ and so, from the inequality \eqref{8_30_3}, $$x_{n_{i}-3}\rightarrow \infty.$$ Thus, for sufficiently large $n_{i}$, it follows from the inequality \eqref{8_30_4} that
$$x_{n_{i}}<x_{n_{i}-3},$$ which contradicts our earlier assumption.
\end{proof}
\begin{corollary}
Every solution of the system \#$(8,30)$ is bounded.
\end{corollary}
\begin{proof}
Let $\{(x_{n},y_{n})\}^{\infty}_{n=1}$ be a solution of the system \#$(8,30)$. From Theorem \ref{thm3} and the earlier reduction we know that $\{x_{n}\}^{\infty}_{n=1}$ is bounded.
Since $$y_{n}=x_{n+1}x_{n},\quad n\geq 0,$$ for all solutions of the system \#$(8,30)$, we get that $\{y_{n}\}^{\infty}_{n=1}$ is bounded.
\end{proof}

\section{Conclusion}
We have shown that the unique equilibrium of the difference equation \#166 is globally asymptotically stable and we have shown that every solution of the difference equation \#$(8,30)$ is bounded. We leave the reader with three crucial conjectures pertaining to four special cases.\par
The special cases \#68 and \#141 are the two subcases of the following second-order rational difference equation 
$$x_{n+1}=\frac{\alpha + x_{n}}{A+ x_{n}+ Cx_{n-1}},\quad n=0,1,\dots ,$$
with $A\geq 0$, all other parameters positive, and nonnegative initial conditions. It is conjectured in \cite{klbook} that the unique positive equilibrium is globally asymptotically stable for the difference equations \#68 and \#141. The difference equations \#68 and \#141 are now the only second-order rational difference equations for which the qualitative behavior has not been established yet. \par
The special case \#70 is the only remaining third-order rational difference equation whose boundedness character is yet to be determined. Special case \#70 is as follows,
$$x_{n+1}=\frac{\alpha + x_{n}}{Cx_{n-1}+x_{n-2}},\quad n=0,1,\dots ,$$
with positive parameters and nonnegative initial conditions. It is conjectured in \cite{clbook} that there exist unbounded solutions for some choice of nonnegative initial conditions for the difference equation \#70.\par
The special case \#$(6,25)$ is the following system of rational difference equations,
$$x_{n+1}=\frac{x_{n}}{y_{n}},\quad y_{n+1}= x_{n} + \epsilon y_{n},\quad n=0,1,2,\dots .$$
with positive parameters and nonnegative initial conditions. There is a conjecture in \cite{aclr}, originating in \cite{ltt}, which claims that for each solution of the system \#$(6,25)$ the sequence $\{x_{n}\}^{\infty}_{n=0}$ arising from the $x$ component of the solution is bounded. The system \#$(6,25)$ is now the only rational system in the plane for which the boundedness character has not been established yet. \par

\par
\end{document}